\documentclass[11pt]{amsart}

\usepackage[T1]{fontenc}
\usepackage[utf8]{inputenc}
\usepackage{geometry}
\usepackage{amssymb}
\usepackage{amsmath}
\usepackage{mathtools}
\usepackage{xcolor}
\usepackage[colorlinks=true,citecolor=red,linkcolor=blue,urlcolor=red]{hyperref}

\DeclareMathOperator{\conv}{conv}
\DeclareMathOperator{\Vertt}{Vert}
\DeclareMathOperator{\lct}{lct}
\DeclareMathOperator{\barc}{bar}

\newcommand{\RR}{\mathbb{R}}
\newcommand{\ZZ}{\mathbb{Z}}
\newcommand{\QQ}{\mathbb{Q}}
\newcommand{\PP}{\mathbb{P}}

\newtheorem{thm}{Theorem}[section]

\newtheorem{prop}[thm]{Proposition}

\newtheorem{conj}[thm]{Conjecture}

\theoremstyle{definition}

\newtheorem{ques}[thm]{Question}
\newtheorem{rem}[thm]{Remark}

\numberwithin{equation}{section}

\begin{document}

\title[Sharpness of Tian's criterion for K-stability]{K-polystable toric Fano varieties with small alpha invariants}

\author{Jihao Liu}
\address{Department of Mathematics, Peking University, No. 5 Yiheyuan Road, Haidian District, Beijing 100871, China}
\address{Beijing International Center for Mathematical Research, Peking University, No. 5 Yiheyuan Road, Haidian District, Beijing 100871, China}
\email{liujihao@math.pku.edu.cn}

\author{Ziwen Zhu}
\address{Department of Mathematics, Tongji University, No. 1239 Siping Road, Yangpu District, Shanghai 200092, China}
\email{zzhu@tongji.edu.cn}

\subjclass[2020]{14J45, 32Q20, 14M25, 52B20}
\keywords{Tian's criterion, alpha invariant, K-stability, K-semistable Fano variety, toric variety, reflexive polytope}
\date{}

\begin{abstract}
For every $n\geq 2$, we exhibit an $n$-dimensional K-polystable toric $\QQ$-Fano variety $X_n$, defined by the face fan of an explicit lattice polytope, and whose alpha invariant is exactly $\tfrac{2}{2n+1}$. This answers a question of Liu and Zhuang whether there exists an $n$-dimensional K-semistable $\QQ$-Fano variety whose alpha invariant is between $\tfrac{1}{n+1}$ and $\tfrac1n$. The main result of this paper was obtained by Chatgpt 5.5 pro, and the Danus system based on the Rethlas system.
\end{abstract}

\maketitle

\section{Introduction}\label{sec:introduction}
Let $X$ be a $\QQ$-Fano variety. 
The alpha invariant (also known as global log canonical threshold) of $X$, denoted by $\alpha(X)$, measures the worst singularities in the anti-pluricanonical systems, and is defined by
\[
\alpha(X)=\inf\{\lct(X;D)|{D\sim_{\QQ}-K_X,\ D\geq 0}\}.
\]

Tian's criterion, formulated in terms of the alpha invariant, is one of the foundational sufficient conditions for K-stability of Fano varieties: an $n$-dimensional Fano variety with alpha invariant  greater than (or no less than) $\tfrac{n}{n+1}$ is K-(semi)stable. Nevertheless, there exist many examples of K-semistable Fano varieties whose alpha invariants fall below $\tfrac{n}{n+1}$. The following theorem in \cite{FO16} establishes a lower bound for the alpha invariants of K-semistable Fano varieties. 
\begin{thm}[\cite{FO16}]
Let $X$ be a K-semistable $\QQ$-Fano variety of dimension $n$. Then $\alpha(X)\geq \tfrac{1}{n+1}$.
\end{thm}
This lower bound is sharp, as the $n$-dimensional projective space $\PP^n$ attains precisely $\tfrac{1}{n+1}$ as its alpha invariant. Moreover, it is shown in \cite{Jia17} that $\PP^n$ is the only $n$-dimensional K-semistable Fano manifold achieving this minimum value of alpha invariants. Since the alpha invariant is lower semicontinuous in a family due to \cite{BL22}, a natural question is whether we can identify the second smallest alpha invariant among all K-semistable Fano varieties. The following conjecture was first proposed by Jiang in \cite{Jia17}.
\begin{conj}[{\cite[Conjecture 1.6]{Jia17}}]\label{conj:jiang}
    Let $X$ be a K-semistable Fano manifold.
Then $\alpha(X)<\frac{1}{n}$
if and only if $X\cong \PP^n$.
\end{conj}

This conjecture was later reformulated in \cite{LZ22} as the following question.

\begin{ques}[{\cite[Question~1.5(1)]{LZ22}}]\label{ques:lz}
Let $n\geq 2$ be an integer. Does there exist a K-semistable $\QQ$-Fano variety $X$ of dimension $n$ such that
\[
\frac{1}{n+1}<\alpha(X)<\frac1n\ ?
\]
\end{ques}

The purpose of this note is to answer Question~\ref{ques:lz} affirmatively for every $n\geq 2$, by an explicit family of toric varieties whose alpha invariant lands strictly inside the prescribed interval at the single value $\tfrac{2}{2n+1}$.

\begin{thm}\label{thm:main}
For every integer $n\geq 2$ there exists an $n$-dimensional toric $\QQ$-Fano variety $X_n$ which is K-polystable, hence K-semistable, and satisfies
\begin{equation}\label{eq:alpha-value}
\alpha(X_n)=\frac{2}{2n+1}.
\end{equation}
\end{thm}

We note that the toric examples constructed in this note are singular. Consequently, Conjecture \ref{conj:jiang} may still hold. However, Theorem \ref{thm:main} naturally leads to the following new question.
\begin{ques}\label{ques:new}
Let $n\geq 2$ be an integer. Does there exist a K-semistable $\QQ$-Fano variety $X$ of dimension $n$ such that
\[
\frac{1}{n+1}<\alpha(X)<\frac{2}{2n+1}\ ?
\]
\end{ques}

The construction in this note is purely combinatorial. We work in a pair of dual rank-$n$ lattices arising from the standard hyperplane $\sum x_i=0$ in $\RR^{n+1}$ and its dual quotient. In the character space we form a rational polytope $P_n$ as the convex hull of two cyclically indexed families of vectors $a_i$ and $b_i$, and we let $Q_n=P_n^\vee$ be its polar. Passing to cyclic difference coordinates turns $Q_n$ into a transparent hyperplane slice of a coordinate box, whose vertices, lattice points, and extremal pairings with $P_n$ can be read off directly (Section~\ref{sec:polytope}). The variety $X_n$ is then the toric variety of the face fan of $Q_n$: its rays are generated by the vertices of $Q_n$, and its anticanonical polytope is $P_n$ (Section~\ref{sec:variety}). The cyclic symmetry forces the barycenter of $P_n$ to be the origin, which by the toric criterion of Berman gives K-polystability; the maximum pairing $\max_{u\in P_n,\,v\in\Vertt(Q_n)}\langle u,v\rangle=\tfrac{2n-1}{2}$, fed into the toric formula for the alpha invariant, gives \eqref{eq:alpha-value} (Section~\ref{sec:kstab-alpha}).

\begin{rem}\label{rem:rethlas}
The sketch of the proof of the main result of this paper was obtained by Chatgpt 5.5 pro, and later summed up, verified, and properly written by the Danus system, a specialized agent built on Rethlas and substantially more capable for fundamental mathematical research
based on the Rethlas system. Human verification and polishing were done afterwards. See \cite{Ju+26} for a detailed introduction to the Rethlas system. Due to the limitation of automated systems, it is possible that we have missed some related references in the literature, and we welcome any comments from experts.
\end{rem}

\subsection*{Acknowledgements}
The first author was partially supported by the National Key R\&D Program of China \#\allowbreak 2024YFA1014400.
The first author would like to thank the Rethlas team, namely Haocheng Ju, Jiedong Jiang, Shurui Liu, Guoxiong Gao, Yuefeng Wang, Zeming Sun, Bin Wu, Liang Xiao, and Bin Dong, for their contributions to the development of Rethlas and its customized version used for the problem studied in this paper.
The first author would like to thank Ruochuan Liu and Gang Tian for constant support and encouragement.

\section{The dual polytopes}\label{sec:polytope}

Throughout, fix an integer $n\geq 2$ and set $\ell=2n-1$. Indices are read modulo $n+1$. Let $e_0,\dots,e_n$ be the standard coordinate vectors of $\RR^{n+1}$, and put
\[
M_\RR=\Bigl\{x=(x_0,\dots,x_n)\in\RR^{n+1}:\sum_{i=0}^n x_i=0\Bigr\},
\qquad
N_\RR=\RR^{n+1}/\RR(1,\dots,1),
\]
paired by $\langle x,[c]\rangle=\sum_{i=0}^n x_i c_i$. These are dual real vector spaces of dimension $n$, the underlying lattices being
\[
M=\Bigl\{x\in\ZZ^{n+1}:\sum_{i=0}^n x_i=0\Bigr\},
\qquad
N=\ZZ^{n+1}/\ZZ(1,\dots,1).
\]
In $M_\RR$ define
\begin{equation}\label{eq:ai-bi}
a_i=\frac{e_i-e_{i+1}}{2},
\qquad
b_i=\frac{e_i-e_{i-1}}{\ell},
\qquad 0\leq i\leq n,
\end{equation}
and set
\begin{equation}\label{eq:Pn}
P_n=\conv\bigl(\{a_i:0\leq i\leq n\}\cup\{b_i:0\leq i\leq n\}\bigr).
\end{equation}
Let
\begin{equation}\label{eq:Qn}
Q_n=P_n^\vee=\{y\in N_\RR:\langle u,y\rangle\geq-1\text{ for every }u\in P_n\}
\end{equation}
be the polar polytope of $P_n$. For a class $[c]\in N_\RR$ we record the \emph{cyclic difference coordinates}
\[
d_i=c_{i+1}-c_i,\qquad 0\leq i\leq n,
\]
which satisfy $\sum_{i=0}^n d_i=0$ by telescoping. For an ordered pair $p\neq q$ in $\{0,\dots,n\}$ let $v_{p,q}\in N_\RR$ be the class whose difference coordinates are
\begin{equation}\label{eq:vpq}
d_p=-\ell,\qquad d_q=1,\qquad d_i=2\quad(i\notin\{p,q\}).
\end{equation}

The following proposition is the finite combinatorial core of the construction.

\begin{prop}\label{prop:polytope}
With the notation above, the following hold.
\begin{enumerate}
\item[(1)] $P_n$ is full-dimensional in $M_\RR$ and contains $0$ in its interior.
\item[(2)] The difference-coordinate map $[c]\mapsto d$, with $d_i=c_{i+1}-c_i$, is a linear isomorphism
\[
N_\RR\xrightarrow{\ \sim\ }\Bigl\{d\in\RR^{n+1}:\sum_{i=0}^n d_i=0\Bigr\}
\]
carrying $N$ onto $\{d\in\ZZ^{n+1}:\sum_i d_i=0\}$; under it,
\begin{equation}\label{eq:Qn-box}
Q_n\cong\Bigl\{d\in\RR^{n+1}:\sum_{i=0}^n d_i=0\text{ and }-\ell\leq d_i\leq 2\text{ for all }i\Bigr\}.
\end{equation}
\item[(3)] The vertices of $Q_n$ are exactly the points $v_{p,q}$ of \eqref{eq:vpq}, indexed by the ordered pairs $p\neq q$ in $\{0,\dots,n\}$; there are $n(n+1)$ of them.
\item[(4)] Every $v_{p,q}$ is a primitive lattice point of $N$.
\item[(5)] The barycenter of $P_n$ is $0$.
\item[(6)] $\displaystyle\max\{\langle u,v\rangle:u\in P_n,\ v\in\Vertt(Q_n)\}=\frac{\ell}{2}=\frac{2n-1}{2}$.
\end{enumerate}
\end{prop}

\begin{proof}
\emph{(1).} The vectors $e_i-e_{i+1}$, $0\leq i\leq n$, span the hyperplane $M_\RR$: their integer span contains $e_i-e_0$ for $1\leq i\leq n$, and these differences span the subspace of vectors with zero coordinate sum. Hence the vectors $a_i$ span $M_\RR$, and they affinely span it as well. Moreover
\[
\sum_{i=0}^n a_i=\frac12\sum_{i=0}^n(e_i-e_{i+1})=0,
\]
so $0$ is the average $\tfrac1{n+1}\sum_i a_i$ of the $n+1$ points $a_i$ with all coefficients positive. Since the $a_i$ affinely span $M_\RR$, the origin lies in the relative interior of $\conv\{a_0,\dots,a_n\}$, hence in the interior of $P_n$ in $M_\RR$. Thus $P_n$ is full-dimensional and contains $0$ in its interior.

\emph{(2).} The map $[c]\mapsto d$ with $d_i=c_{i+1}-c_i$ is well defined on $N_\RR$, since adding the same scalar to all $c_i$ changes no difference, and its image lies in $\{d:\sum_i d_i=0\}$ by telescoping. It is injective: if all differences vanish then all $c_i$ are equal, so $[c]=0$ in the quotient. It is surjective: given $d$ with $\sum_i d_i=0$, set $c_0=0$ and $c_{j+1}=c_j+d_j$ for $0\leq j\leq n-1$; the zero-sum condition then yields $d_n=c_0-c_n$, so the class $[c]$ maps to $d$. The same construction over $\ZZ$ identifies $N$ with $\{d\in\ZZ^{n+1}:\sum_i d_i=0\}$.

For $y=[c]$ we compute, from \eqref{eq:ai-bi},
\[
\langle a_i,y\rangle=\frac{c_i-c_{i+1}}{2}=-\frac{d_i}{2},
\qquad
\langle b_i,y\rangle=\frac{c_i-c_{i-1}}{\ell}=\frac{d_{i-1}}{\ell}.
\]
Hence $\langle a_i,y\rangle\geq-1$ is equivalent to $d_i\leq 2$, and $\langle b_i,y\rangle\geq-1$ is equivalent to $d_{i-1}\geq-\ell$. Since $P_n$ is the convex hull of the points $a_i$ and $b_i$, the defining inequalities $\langle u,y\rangle\geq-1$ for all $u\in P_n$ reduce to these inequalities for all $a_i$ and $b_i$. As $i$ ranges over $\{0,\dots,n\}$ this gives exactly $-\ell\leq d_i\leq 2$ for all $i$, together with $\sum_i d_i=0$, which is \eqref{eq:Qn-box}.

\emph{(3).} Let $H=\{d\in\RR^{n+1}:\sum_i d_i=0\}$, an $n$-dimensional hyperplane; the polytope \eqref{eq:Qn-box} is $H$ intersected with the box $[-\ell,2]^{n+1}$. At a vertex of this polytope at least $n$ independent coordinate-bound inequalities are active, so at least $n$ of the coordinates equal one of the endpoints $-\ell$ or $2$. Not all $n+1$ coordinates can be endpoints: if $k$ coordinates equal $2$ and the remaining $n+1-k$ equal $-\ell$, the zero-sum condition gives
\[
2k-\ell(n+1-k)=0,
\qquad\text{equivalently}\qquad
(\ell+2)k=\ell(n+1).
\]
Since $\ell=2n-1$ and $\ell+2=2n+1$, we have $\gcd(\ell,\ell+2)=\gcd(2n-1,2)=1$, so $2n+1$ would have to divide $n+1$; this is impossible because $0<n+1<2n+1$ for $n\geq 2$. Therefore exactly $n$ coordinates are endpoints and one coordinate is free.

Suppose that among the $n$ endpoint coordinates $k$ equal $2$ and $n-k$ equal $-\ell$. The free coordinate is then forced by $\sum_i d_i=0$ to be
\[
t=\ell(n-k)-2k=\ell n-(\ell+2)k.
\]
If $k\leq n-2$, then $t\geq\ell n-(\ell+2)(n-2)=2n+2>2$, exceeding the upper bound. If $k=n$, then $t=\ell n-(\ell+2)n=-2n<-\ell$, below the lower bound. The only admissible case is $k=n-1$, giving $t=\ell-2(n-1)=1\in[-\ell,2]$. Hence every vertex has exactly one coordinate equal to $-\ell$, exactly one equal to $1$, and all remaining coordinates equal to $2$; that is, every vertex is one of the points $v_{p,q}$ of \eqref{eq:vpq}.

Conversely, for each ordered pair $p\neq q$ the point $v_{p,q}$ has zero coordinate sum, since
\[
-\ell+1+2(n-1)=-(2n-1)+1+2n-2=0,
\]
and it satisfies the box constraints. The $n$ active equalities $d_p=-\ell$ and $d_i=2$ for $i\notin\{p,q\}$ are independent on $H$ and determine $v_{p,q}$ uniquely there, so $v_{p,q}$ is a vertex. There are $n(n+1)$ ordered pairs $p\neq q$, giving $n(n+1)$ vertices.

\emph{(4).} By the lattice identification of part (2), each $v_{p,q}$ is a lattice point of $N$. Because $v_{p,q}$ has the coordinate $d_q=1$, it cannot be a positive integer multiple $mw$ of another lattice point $w$ with $m\geq 2$, since then every coordinate would be divisible by $m$. Hence each $v_{p,q}$ is primitive.

\emph{(5).} Let $\sigma$ be the cyclic permutation $\sigma(x_0,\dots,x_n)=(x_1,\dots,x_n,x_0)$ of coordinates. It preserves $M_\RR$, and by \eqref{eq:ai-bi} it sends $a_i\mapsto a_{i+1}$ and $b_i\mapsto b_{i+1}$, so $\sigma(P_n)=P_n$. Being induced by a coordinate permutation, $\sigma$ preserves Lebesgue measure on $M_\RR$; therefore it fixes the barycenter $\barc(P_n)$. The fixed subspace of $\sigma$ in $M_\RR$ consists of the vectors with all coordinates equal, intersected with $\{x:\sum_i x_i=0\}$, which is $\{0\}$. Hence $\barc(P_n)=0$.

\emph{(6).} Fix a vertex $v=v_{p,q}$ of $Q_n$, written in difference coordinates $d$. By the formulas in part (2),
\[
\langle a_i,v\rangle=-\frac{d_i}{2},
\qquad
\langle b_i,v\rangle=\frac{d_{i-1}}{\ell}.
\]
As $i,p,q$ vary, the largest value of $\langle a_i,v\rangle$ is $\ell/2$, attained when $d_i=-\ell$, i.e.\ $i=p$. The largest value of $\langle b_i,v\rangle$ is $2/\ell$, since every difference coordinate is at most $2$. Because $n\geq 2$ we have $\ell\geq 3$ and hence $2/\ell<\ell/2$. For each fixed $v$ the function $u\mapsto\langle u,v\rangle$ is linear on $P_n=\conv(\{a_i\}\cup\{b_i\})$, so its maximum over $P_n$ is attained at one of the generating points $a_i$ or $b_i$. Taking the maximum also over all vertices $v_{p,q}$ gives
\[
\max\{\langle u,v\rangle:u\in P_n,\ v\in\Vertt(Q_n)\}=\frac{\ell}{2}=\frac{2n-1}{2}. \qedhere
\]
\end{proof}

\section{The toric variety \texorpdfstring{$X_n$}{Xn}}\label{sec:variety}

Let $\Sigma_n$ be the face fan of the polytope $Q_n$, that is,
\[
\Sigma_n=\{\RR_{\geq0}F:F\text{ is a face of }Q_n\}.
\]
Since $0$ lies in the interior of $Q_n$ by Proposition~\ref{prop:polytope}(1) applied to its polar, $\Sigma_n$ is a complete fan in $N_\RR$. Let $X_n=X_{\Sigma_n}$ be the associated $n$-dimensional normal projective toric variety.

By Proposition~\ref{prop:polytope}(3),(4) the primitive generators of the rays of $\Sigma_n$ are exactly the vertices $v_{p,q}$ of $Q_n$. Writing $D_{p,q}$ for the torus-invariant prime divisor associated with the ray through $v_{p,q}$, the anticanonical divisor of a toric variety is the sum of all torus-invariant prime divisors, so
\[
-K_{X_n}=\sum_{p\neq q}D_{p,q}.
\]
The polytope of $-K_{X_n}$ is the set of characters $u$ with $\langle u,v\rangle\geq-1$ for every ray generator $v$, which by definition of the polar is
\[
P_{-K_{X_n}}=\{u\in M_\RR:\langle u,v_{p,q}\rangle\geq-1\text{ for all }p\neq q\}=Q_n^\vee=P_n.
\]
Each facet of $Q_n$ has a supporting character $u_F\in M_\QQ$ with $\langle u_F,v\rangle=-1$ for $v$ in the facet; this is the local Cartier datum of $-K_{X_n}$ on the corresponding maximal cone, so $-K_{X_n}$ is $\QQ$-Cartier. As $P_n=P_{-K_{X_n}}$ is bounded and full-dimensional by Proposition~\ref{prop:polytope}(1), the divisor $-K_{X_n}$ is ample. Finally, on the cone over a facet $F$ any nonzero vector $w=\sum_j\lambda_j v_j$ with $\lambda_j\geq0$ has positive log discrepancy
\[
A_{X_n}(w)=-\langle u_F,w\rangle=\sum_j\lambda_j>0,
\]
so $X_n$ is klt. We summarize.

\begin{prop}\label{prop:qfano}
For every integer $n\geq 2$, the variety $X_n=X_{\Sigma_n}$ is an $n$-dimensional toric $\QQ$-Fano variety, with $-K_{X_n}$ ample and anticanonical polytope $P_{-K_{X_n}}=P_n$.
\end{prop}

\section{K-polystability and the alpha invariant}\label{sec:kstab-alpha}

We now deduce K-polystability and compute the alpha invariant from the polytope data of Proposition~\ref{prop:polytope}, and prove Theorem~\ref{thm:main}.

\begin{prop}\label{prop:kss}
For every integer $n\geq 2$, the toric $\QQ$-Fano variety $X_n$ is K-polystable, hence K-semistable.
\end{prop}

\begin{proof}
By Berman's toric criterion \cite[Corollary~1.2]{Ber16}, a toric $\QQ$-Fano variety is K-polystable if and only if the barycenter of its anticanonical weight polytope is the origin. The anticanonical polytope of $X_n$ is $P_n$ by Proposition~\ref{prop:qfano}, and $\barc(P_n)=0$ by Proposition~\ref{prop:polytope}(5). Hence $X_n$ is K-polystable, and in particular K-semistable.
\end{proof}

\begin{prop}\label{prop:alpha}
For every integer $n\geq 2$, the alpha invariant of $X_n$ is
\[
\alpha(X_n)=\frac{2}{2n+1}.
\]
\end{prop}

\begin{proof}
For a toric $\QQ$-Fano variety with anticanonical polytope $P$ and set of primitive ray generators $V$, the ordinary alpha invariant is given by the toric formula
\begin{equation}\label{eq:toric-alpha}
\alpha(X)=\frac{1}{\displaystyle\max_{u\in P,\,v\in V}\bigl(1+\langle u,v\rangle\bigr)}
\end{equation}
of Blum and Jonsson \cite[Corollary~7.16]{BJ17}. For $X_n$ the polytope is $P=P_n$ and the ray generators are $V=\Vertt(Q_n)$ by Propositions~\ref{prop:polytope} and~\ref{prop:qfano}. By Proposition~\ref{prop:polytope}(6),
\[
\max_{u\in P_n,\,v\in\Vertt(Q_n)}\langle u,v\rangle=\frac{2n-1}{2},
\]
so the maximum in \eqref{eq:toric-alpha} equals $1+\tfrac{2n-1}{2}=\tfrac{2n+1}{2}$. Therefore
\[
\alpha(X_n)=\frac{1}{\tfrac{2n+1}{2}}=\frac{2}{2n+1}. \qedhere
\]
\end{proof}

\begin{proof}[Proof of Theorem~\ref{thm:main}]
Fix $n\geq 2$. By Proposition~\ref{prop:qfano}, $X_n$ is an $n$-dimensional toric $\QQ$-Fano variety. By Proposition~\ref{prop:kss} it is K-polystable, hence K-semistable, and by Proposition~\ref{prop:alpha} its alpha invariant equals $\tfrac{2}{2n+1}$, which is \eqref{eq:alpha-value}. 
\end{proof}

\begin{rem}\label{rem:scope}
The construction answers part (1) of \cite[Question~1.5]{LZ22} for every $n\geq 2$, and K-polystability is stronger than the K-semistability requested. We make no claim regarding part (2) of that question.
\end{rem}




\bibliographystyle{alpha}
\bibliography{ref}

\end{document}